\documentclass[reqno]{amsart}

\usepackage{amsmath,amssymb,amsthm}

\usepackage{graphicx}

\usepackage{xcolor}

\usepackage{float} 
\usepackage{lscape}

\numberwithin{equation}{section}

\usepackage{amsmath,amssymb,amsthm,mathrsfs}
\usepackage{longtable}
\usepackage{enumerate}

\renewcommand{\S}{{\mathbb S}}
\renewcommand{\P}{{\mathbb P}}
\newcommand{\R}{{\mathbb R}}
\newcommand{\C}{{\mathbb C}}
\newcommand{\HH}{{\mathbb H}}
\newcommand{\OO}{{\mathbb O}}
\newcommand{\M}{{\mathcal M}}
\newcommand{\V}{\operatorname{vol}}
\allowdisplaybreaks

\theoremstyle{plain}
\newtheorem{theorem}{Theorem}[section]
\newtheorem{lem}[theorem]{Lemma}
\newtheorem{prop}[theorem]{Proposition}

\begin{document}

\title[Lower bound for the Green energy in harmonic manifolds]{Lower bound for the Green energy of point configurations in harmonic manifolds}

\author[C. Beltr\'an]{Carlos Beltr\'an}
\address{Departamento de Matem\'aticas, Estad\'istica y Computaci\'on, Universidad de Cantabria. 39005. Santander, Spain}
\email{beltranc@unican.es}

\author[V. de la Torre]{V\'ictor de la Torre}
\address{Departament de Matem\`atiques I Inform\`atica, Universitat de Barcelona. 08007. Barcelona, Spain}
\email{delatorre@ub.edu}

\author[F. Lizarte]{F\'atima Lizarte}
\address{Departamento de Matem\'aticas, Estad\'istica y Computaci\'on,
Universidad de Cantabria. 39005. Santander, Spain}
\email{fatima.lizarte@unican.es}

\keywords{Green energy, compact harmonic manifolds}

\subjclass{Primary: 31C12, 31C20, 41A60}

\thanks{First and third authors belong to the Universidad de Cantabria and are supported by Grant PID2020-113887GB-I00 funded by MCIN/ AEI /10.13039/501100011033. The second author belongs to the Universitat de Barcelona and has been partially supported by grant MTM2017-83499-P by the Ministerio de Ciencia, Innovaci\'on y Universidades, Gobierno de Espa\~{n}a and by the Generalitat de Catalunya (project 2017 SGR 358). The third author has also been supported by Grant PRE2018-086103 funded by MCIN/AEI/10.13039/501100011033 and by ESF Investing in your future}

\begin{abstract}
In this paper, we get the sharpest known to date lower bounds for the minimal Green energy of the compact harmonic manifolds of any dimension.    
\end{abstract}

\maketitle

\section{Introduction}

\subsection*{The Green function.}
Let $\M$ be any compact Riemannian manifold. The Green function $G(\M;\cdot,\cdot)$ is the unique function $G:(\M\times\M)\setminus\{(p,p):p\in\M\}\to\R$ with the properties:
\begin{enumerate}
\item In the sense of distributions, $\Delta_q G=\mathcal{S}_p(q)-\text{vol}(\mathcal{M})^{-1}$, where $\mathcal{S}_p$ is Dirac's delta and $\Delta=-\mathrm{div}\nabla$ is the Laplace--Beltrami operator, which is the natural extension of the Laplacian to $\M$ (note the sign convention).
\item Symmetry: $G(\mathcal{M};p,q)=G(\mathcal{M};q,p)$.
\item The mean of $G(\mathcal{M};p,\cdot)$ is zero for all $p\in\M$, i.e., $\int_{q\in\mathcal{M}}G(\mathcal{M};p,q)dq=0$.
\end{enumerate}
\subsection*{The Green energy.}
Let $p_1,\ldots,p_N\in\M$ and consider the Green energy
$$
E_{\mathcal{M}}(p_1, \ldots, p_N)=\sum_{i\neq j}G(\M;p_i,p_j).
$$
The search for minimizers of the Green energy is an interesting and difficult mathematical problem. 
If $\mathcal M=\S^2$ is the usual $2$--sphere, we have
\begin{equation}\label{eq:Greenlog}
G({\S}^2;p,q)=\frac{1}{2\pi}\log\frac{1}{\|p-q\|}-\frac{1}{4\pi}+\frac{\log 2}{2\pi},
\end{equation}
where $\log $ denotes the natural logarithm. Hence, the search for minimizers of the Green energy in $\S^2$ is the question of Smale's 7th problem \cite{Smale2000}. 

In a general compact Riemannian manifold, if $p_1,\ldots,p_N$ are minimizers of the Green energy for increasing values of $N$, then they are asymptotically uniformly distributed, i.e., the associated counting probability measure converges in the weak sense to the uniform probability measure in $\mathcal M$, see \cite{BeltranCorralCriado2019}. More quantitatively, in \cite{Stefan2} it is shown that the Wasserstein $2$--distance between these two measures is of order $N^{-1/\mathrm{dim}(\mathcal M)}$, which is the best possible for dimension greater than or equal to $3$. Here and all along the paper, $\mathrm{dim}(\mathcal M)$ stands for the real dimension of a manifold $\mathcal M$.

\subsection*{Minimal value of the Green energy in spheres.}
Upper and lower bounds for the least possible Green energy have been investigated by several authors. The most studied case is that of $\S^2$. After \cite{Wagner89,RSZ94,Dubickas96,Brauchart08,BS18,Stefan} it is known that
$$
\min_{p_1,\ldots,p_N\in\S^2}\sum_{i\neq j}\log\frac{1}{\|p_i-p_j\|}= \left(\frac12-\log2\right)N^2-\frac12N\log N+C_{\mathrm{log}}N+o(N),
$$
where $C_{\log}$ is a constant whose value is not known. From \cite{BS18} we have
$$
C_{\mathrm{log}}\leq C_{BHS}= 2\log2+\frac12\log\frac23+3\log\frac{\sqrt\pi}{\Gamma(1/3)}=-0.0556\ldots
$$
This upper bound has been conjectured to be an equality using several different approaches \cite{BHS12,BS18,Stefan}; see also \cite{BHS19} for context and history of these results. The best currently known lower bound \cite{Lauritsen21} has the same form but for a slightly different constant $\log2-\frac34=-0.0568\ldots$ instead of $C_{\mathrm{log}}$. These bounds can be translated using \eqref{eq:Greenlog}
in terms of the Green energy:
\begin{multline}\label{eq:lbLaur}
-\frac{1}{8\pi}N+o(N)\leq\min_{p_1,\ldots,p_N\in\S^2}E_{\mathcal{\S}^2}(p_1, \ldots, p_N)+\frac1{4\pi}N\log N\leq\\ \frac{1}{4\pi}\left(2C_{BHS}+1-2\log2\right)N+o(N)=
-\frac{0.9950\ldots}{8\pi}N+o(N).
\end{multline}
It has been proved in \cite{BeltranLizarte22} that, if $\mathcal M=\S^n$ is the $n$--sphere, the argument in \cite[Appendix B]{Lauritsen21} (see \cite{LN75, SM76} for some precedents) can be adapted to get a seemingly sharp lower bound
\begin{equation}\label{eq:lbSn}
\min_{p_1,\ldots,p_N\in\S^n}E_{\S^n}(p_1, \ldots, p_N)\geq-\frac{n^{1+2/n}}{(n^2-4)V_n^{1-2/n}V_{n-1}^{2/n}}N^{2-2/n}+o(N^{2-2/n}),
\end{equation}
where $V_n=2\pi^{(n+1)/2}/\Gamma((n+1)/2)$ is the volume of $\S^n$. Upper bounds of the same order, also with explicit constants, can be obtained from the respective bounds for Riesz energies, see \cite{BMOC} and follow--up papers.
\subsection*{Minimal value of the Green energy in general manifolds.}
In a general compact Riemannian manifold, \cite[p. 4, Corollary]{Stefan2} proved that
$$
\min_{p_1,\ldots,p_N\in\mathcal{M}}E_{\mathcal{M}}(p_1, \ldots, p_N)\geq\begin{cases}\mathrm{Constant}(\mathcal M)N\log N&\mathrm{dim}(\mathcal M)=2,\\
\mathrm{Constant}(\mathcal M)N^{2-2/\mathrm{dim}(\mathcal M)}&\mathrm{dim}(\mathcal M)\geq3.
\end{cases}
$$
It is easy to see that $\mathrm{Constant}(\mathcal M)$ is negative in all cases, but obtaining explicit values for a given $\mathcal M$ seems to be a much more difficult task in general.

\subsection*{Minimal value of the Green energy in harmonic manifolds.}
Recall that the compact harmonic manifolds are the sphere $\S^n$, the real, complex and quaternionic projective spaces $\mathbb{RP}^n,\mathbb{CP}^n,\mathbb{HP}^n$ and the Cayley plane $\mathbb{OP}^2$. These are all $2$--point homogeneous spaces: if $p_1,q_1,p_2,q_2\in\mathcal M$ satisfy $d_R(p_1,q_1)=d_R(p_2,q_2)$, then there exists an isometry of $\mathcal M$ that takes $p_1$ to $p_2$ and $q_1$ to $q_2$. This fact implies that many geometric properties (including minimal energy computations) can be described in a simpler manner than for general manifolds. The case $\mathcal M=\mathbb{RP}^2$ is particularly simple since, as noted in \cite{Pedro}, $E_{\mathbb{RP}^2}(p_1,\ldots,p_N)$ can be written in terms of $E_{\S^2}(p_1,\ldots,p_N,-p_1,\ldots,-p_N)$ and the lower bound on the latter implies a lower bound on the former:
\begin{equation}\label{eq:lbRP2}
\min_{p_1,\ldots,p_N\in\mathbb{RP}^2}E_{\mathbb{RP}^2}(p_1, \ldots, p_N)\geq-\frac{N}{4\pi}\log N+\frac{1}{4\pi}\left(\frac12-\log2\right)N+o(N).
\end{equation}
Moreover, $\mathbb{CP}^1$ is isometric to the Riemann sphere, that is, the sphere of radius $1/2$ centered at $(0,0,1/2)$, and hence $E_{\mathbb{CP}^1}(p_1,\ldots,p_N)=4E_{\mathbb{S}^2}(2\hat p_1,\ldots,2\hat p_N)$ for some $\hat p_i$ given by the aforementioned isometry. This implies from \eqref{eq:lbLaur}:
\begin{equation}\label{eq:lbCP1}
\min_{p_1,\ldots,p_N\in\mathbb{CP}^1}E_{\mathbb{CP}^1}(p_1, \ldots, p_N)\geq-\frac{1}{\pi}N\log N-\frac{1}{2\pi}N+o(N).
\end{equation}
These are the sharpest known lower bounds for the harmonic manifolds of real dimension $2$. The higher--dimensional case has been studied in \cite{BeltranEtayo18} for the complex projective space and in \cite{Matzke22} for general harmonic manifolds. This last paper contains the sharpest upper and lower bounds for the Green energy to the date. The notation in that paper is slightly different from ours, since in it the Riemannian metric is normalized in such a way that each $\mathcal M$ has unit volume. Translating their result to our notation, we summarize the lower bounds of \cite{Matzke22}:
\begin{align*}
E_{\R\P^n}(p_1, \ldots, p_N)\geq& -\displaystyle\frac{n}{4(n-2)V}\left(\frac{\sqrt{\pi}}{\Gamma\left(\frac{n+1}{2}\right)}\right)^{2/n}N^{2-2/n}+o(N^{2-2/n}),\\
E_{\C\P^n}(p_1, \ldots, p_N)\geq& -\displaystyle\frac{n}{4(n-1)n!^{1/n}V}N^{2-1/n}+o(N^{2-1/n}),\\
E_{\HH\P^n}(p_1, \ldots, p_N)\geq& -\displaystyle\frac{n}{2(2n-1)\Gamma(2n+2)^{1/2n}V}N^{2-1/2n}+o(N^{2-1/2n}),\\
E_{\OO\P^2}(p_1, \ldots, p_N)\geq& -\displaystyle\frac{2}{7V}\sqrt[8]{\frac{6}{11!}}N^{\frac{15}{8}}+o(N^{\frac{15}{8}}).
\end{align*}
In each case, $V$ holds for the volume of the corresponding manifold, given in Table \ref{table:dov}.
The main goal of this paper is to show that the argument in \cite{Lauritsen21,BeltranLizarte22} can indeed be extended quite straightforwardly to all the harmonic manifolds of any dimension, sharpening the lower bounds for the minimal Green energy:

\begin{theorem}[Main Theorem]\label{th:mainintro}
The following lower bounds for the Green energy of $N$ points in each compact harmonic manifold $\mathcal M$ with $\mathrm{dim}(\mathcal M)>2$ holds:
\begin{align*}
E_{\R\P^n}(p_1, \ldots, p_N)\geq& -\displaystyle\frac{n}{(n^2-4)V}\left(\frac{\Gamma\left(\frac{n}{2}+1\right)\sqrt{\pi}}{\Gamma\left(\frac{n+1}{2}\right)}\right)^{2/n}N^{2-2/n}+o(N^{2-2/n}),\\
E_{\C\P^n}(p_1, \ldots, p_N)\geq& -\displaystyle\frac{n}{2(n^2-1)V}N^{2-1/n}+o(N^{2-1/n}),\\
E_{\HH\P^n}(p_1, \ldots, p_N)\geq& -\displaystyle\frac{n}{(2n-1)(2n+1)^{1+1/2n}V}N^{2-1/2n}+o(N^{2-1/2n}),\\
E_{\OO\P^2}(p_1, \ldots, p_N)\geq& -\displaystyle\frac{4}{63\sqrt[8]{165}V}N^{\frac{15}{8}}+o(N^{\frac{15}{8}}).
\end{align*}
Our method applies equally to $\S^2$, $\S^n$ with $n\geq3$, $\mathbb{RP}^2$ and $\mathbb{CP}^1$, which yields the same lower bounds as in \eqref{eq:lbLaur}, \eqref{eq:lbSn}, \eqref{eq:lbRP2} and \eqref{eq:lbCP1}, respectively. 
\end{theorem}

We can compare our bounds with the ones of \cite{Matzke22} mentioned above, and in all the cases our bounds are sharper, see figures \ref{fig:RPncomparison}, \ref{fig:CPncomparison} and \ref{fig:HPncomparison} for the comparison in the real, complex and quaternionic projective cases and observe that
$$
0.0335\ldots=\frac{4}{63\sqrt[8]{165}}<\frac{2}{7}\sqrt[8]{\frac{6}{11!}}=0.0400\ldots
$$
for the Cayley plane.

\begin{figure}
\centering
\includegraphics[width=1\textwidth]{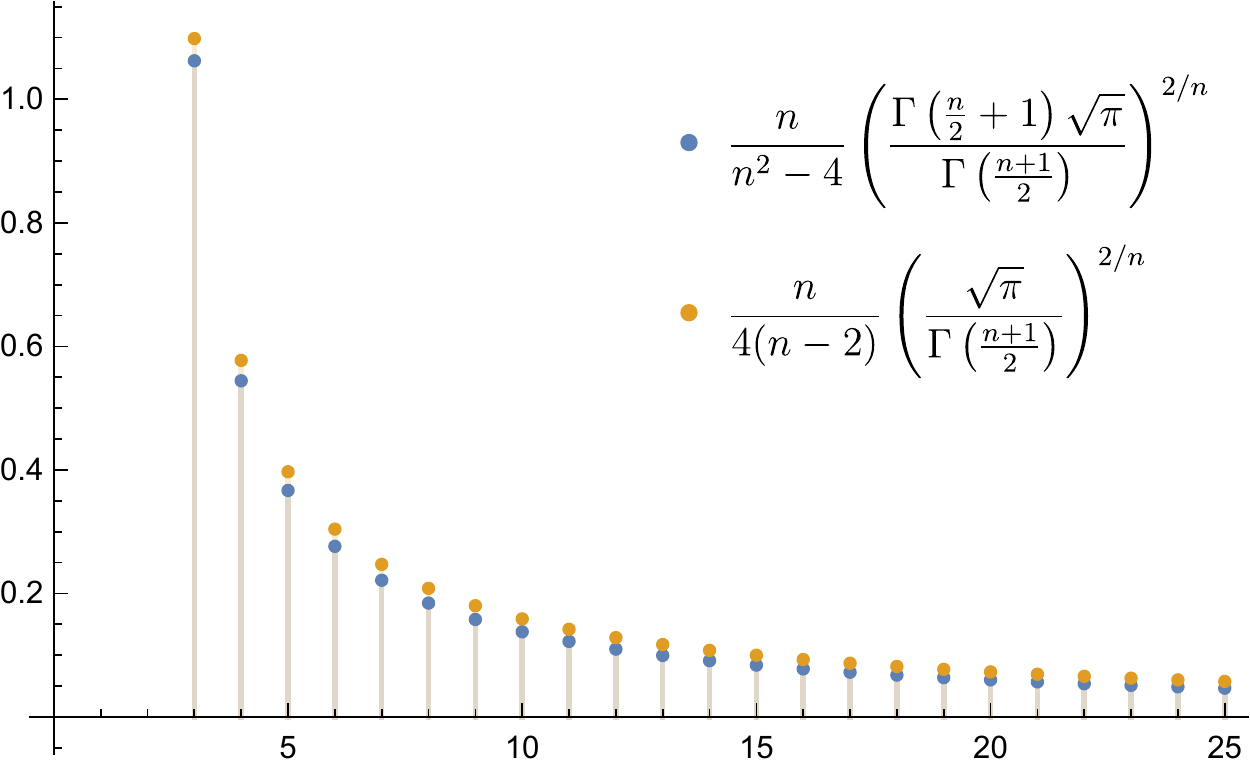}
\caption{The absolute value of the dominant coefficients in the lower bound for $E_{\mathbb{R}\mathbb{P}^{n}}(p_1,\ldots,p_N)$, without the $1/V$ factor and for increasing values of $n$. Blue dots are our constants in Theorem \ref{th:mainintro} and yellow dots are those of \cite{Matzke22}.}
    \label{fig:RPncomparison}
\end{figure}
\begin{figure}
\centering
\includegraphics[width=1\textwidth]{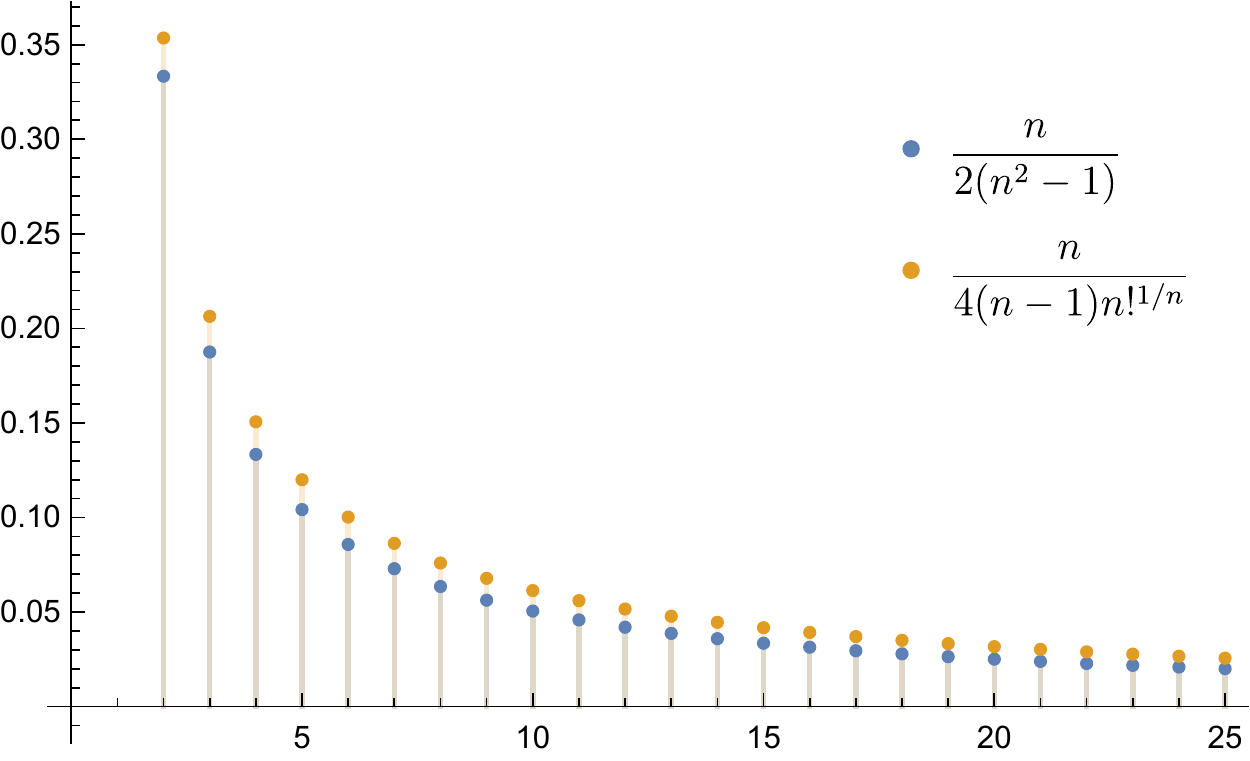}
\caption{The absolute value of the dominant coefficients in the lower bound for $E_{\mathbb{C}\mathbb{P}^{n}}(p_1,\ldots,p_N)$, without the $1/V$ factor and for increasing values of $n$. Blue dots are our constants in Theorem \ref{th:mainintro} and yellow dots are those of \cite{Matzke22}.}
    \label{fig:CPncomparison}
\end{figure}
\begin{figure}
\centering
\includegraphics[width=1\textwidth]{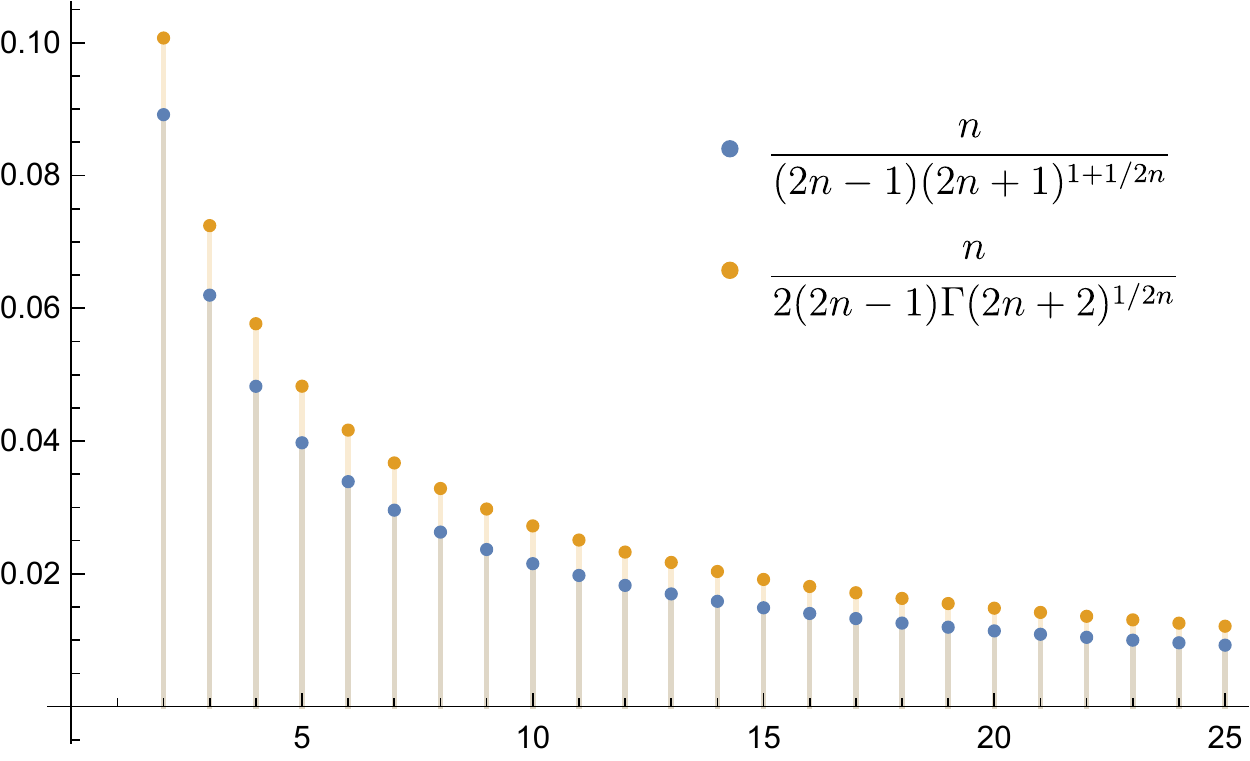}
\caption{The absolute value of the dominant coefficients in the lower bound for $E_{\mathbb{H}\mathbb{P}^{n}}(p_1,\ldots,p_N)$, without the $1/V$ factor and for increasing values of $n$. Blue dots are our constants in Theorem \ref{th:mainintro} and yellow dots are those of \cite{Matzke22}.}
    \label{fig:HPncomparison}
\end{figure}

\section{Harmonic manifolds}
\subsection{Basic definitions and notation}Harmonic manifolds are the most symmetric manifolds that one can conceive. There are just five examples of compact harmonic manifolds (up to dimension choices): $\mathbb S^n, \mathbb R\mathbb P^n, \mathbb C\mathbb P^n, \mathbb H\mathbb P^n$ and $\mathbb O\mathbb P^2$. That is, the $n$--dimensional sphere, the real, complex and quaternionic projective spaces of any dimension, and the octonionic projective space of (octonionic) dimension $2$, that is, real dimension $16$, usually called the Cayley plane.
We will use the following notation:
\begin{itemize}
    \item $d=d_{\mathcal M}=\mathrm{dim}_\R(\mathcal M)$ is the real dimension of the compact harmonic manifold $\mathcal M$.
    \item $D=D_{\mathcal M}$ is the diameter of $\mathcal M$, that is, the maximum Riemannian distance between two points in $\mathcal M$.
    \item $B(p,a)=B_{\mathcal M}(p,a)=\{q\in\mathcal M:d_R(p,q)<a\}$ is the ball centered at $p$ of radius $a$. Here, $d_R$ is the Riemannian distance.
    \item $V(a)=V_{\mathcal M}(a)$ is the volume of the ball $B_{\mathcal M}(p,a)$. Note that due to the symmetry of the harmonic manifolds, this quantity does not depend on $p\in{\mathcal M}$.
    \item $V={V_\mathcal M}=V_{\mathcal M}(D)$ is the volume of $\mathcal M$.
     \item $S(p,a)=S_{\mathcal M}(p,a)=\{q\in\mathcal M:d_R(p,q)=a\}$ is the sphere centered at $p$ of radius $a$.
    \item $v(a)=v_{\mathcal M}(a)$ is the $(d-1)$--dimensional volume of the sphere $S_{\mathcal M}(p,a)$, with the inherited Riemannian structure. Again, this value is independent of $p\in\mathcal M$. 
    \item The exponential map $\mathrm{exp}_{p_0}=\mathrm{exp}_{\mathcal M,p_0}$ is
$$
\begin{matrix}
  \mathrm{exp}_{p_0}: & \{v\in T_{p_0}\mathcal M:\|v\|< D\} & \to &\mathcal M \\
  &v & \to &\mathrm{exp}_{p_0}(v).
\end{matrix}
$$
Here, $p_0$ is any point in $\mathcal M$ 
and $\mathrm{exp}_{p_0}(v)$ is equal to $\gamma_{p_0,v}(t=1)$, with $\gamma_{p_0,v}$ the geodesic passing by $p_0$ with tangent vector $v$ at $t=0$.

    \item $\Omega(r)=\Omega_{\mathcal M}(r)$ is equal to the Jacobian $\text{Jac}(\exp_{p_0})(\exp_{p_0}^{-1}(q))$ for some $p_0,q$ such that $d_R(p_0,q)=r$. This is usually called the volume density function. Since $\mathcal M$ is $2$--point homogeneous, it is independent of the concrete choice of $p_0$ and $q$.
    \item $\mathcal B_{\mathcal M}/V_{\mathcal M}$ is the constant in the first asymptotic term of the Green function for $d\geq3$, that is
    \begin{equation}\label{eq:asymptGreen}
    G(\mathcal M;p,q)=\frac{B_{\mathcal M}}{V_{\mathcal M}\,d_R(p,q)^{d-2}}+O\left(\frac{1}{d_R(p,q)^{d-3}}\right).
    \end{equation}
    In the sphere case it can be obtained from \cite{BeltranLizarte22} by combining Proposition 3.1 and Lemma C.2, while for the projective cases it corresponds to \cite[eq. (2.9)]{Matzke22}.
    \item Finally, we consider two functions that will be useful in our analysis:
    \begin{align}
K(\mathcal M,a)=&\,\,\frac{1}{V\cdot V(a)}\int_0^a v(r)\int_0^r\frac{V(u)}{v(u)}dudr,\label{eq:defK}
\\
\Theta(\mathcal M,a)=&\,\,\frac{1}{ V(a)}\int_{q\in B(p_0,a)}G(\mathcal{M};p_0,q)dq.\label{eq:defTheta}
\end{align}
The first of these two terms appears in the closed formula for the expected value of the Green function in a ball given in Lemma \ref{lem:integralballprev}.
\end{itemize}
Except for the last item, these are all standard definitions in Riemannian geometry. We present the value of these constants and functions for the different choices of $\mathcal M$ in Table 
\ref{table:dov}. 

\subsection{Computing the Green function in harmonic manifolds}\label{section:green}
From the change of variables theorem, for any integrable function $F:\mathcal M\to\R$ such that $F(p)=f(d_R(p,p_0))$ depends only on $d_R(p,p_0)$ we have
\begin{align}
\int_{p\in \mathcal M}F(p)\,dp=&\int_{v\in T_{p_0}\mathcal M:\|v\|< D}F(\exp_{p_0}(v))\Omega(\|v\|)\,dv\nonumber\\
=&\int_0^D\Omega(r)\int_{v\in T_{p_0}\mathcal M:\|v\|= r}f(\|v\|)\,dv\,dr\nonumber
\\
=& \,\V(\mathbb S^{d-1})\int_0^Dr^{d-1}\Omega(r)f(r)\,dr.\label{eq:changeofv}
\end{align}
In particular,
\begin{align}
V(a)=& \,\V(\mathbb{S}^{d-1})\int_0^ar^{d-1}\Omega(r)\,dr.\label{eq:Volbola}
\end{align}
Following \cite{BeltranCorralCriado2019} we have $G(\M;p,q)=\phi(d_R(p,q))$, where
$$ \phi'(r)=-\frac{V^{-1}\int_r^Dt^{d-1}\Omega(t)\,dt}{r^{d-1}\Omega(r)},
$$
which can be computed with Table \ref{table:dov} at hand. We can then integrate $\phi'$ to get the Green function.
The integration constant must be chosen to grant that the integral in $\mathcal{M}$ of $G(\M;p,\cdot)$ is zero for all $p\in\mathcal{M}$. In other words,
\begin{equation}\label{eq:phiCM}
\phi(r)=V^{-1}\left(\hat\phi(r)+C_\M\right),\quad \hat\phi(r)= \int_r^D\frac{\int_s^Dt^{d-1}\Omega(t)\,dt}{s^{d-1}\Omega(s)}\,ds,
\end{equation}
where $C_\M$ is a constant whose value is given in the following result.
\begin{lem}\label{lem:theconstant}
The value of the constant $C_\M$ in \eqref{eq:phiCM} is:
$$
C_\M=-\frac{\V(\S^{d-1})}{V}\int_0^D \hat \phi(r)r^{d-1}\Omega(r)\,dr.
$$
\end{lem}
\begin{proof}
  The integral of $G(\M;p,\cdot)$ equals
 \begin{equation*}
 0=\int_{q\in\M}G(\M;p,q)\,dq\stackrel{\eqref{eq:changeofv}}{=}\V(\S^{d-1})\int_0^D \phi(r)r^{d-1}\Omega(r)\,dr.
 \end{equation*} 
  Since we have $\phi=V^{-1}(\hat\phi+C_{\mathcal M})$ and $\int_0^D r^{d-1}\Omega(r)\,dr\stackrel{\eqref{eq:changeofv}}{=}\frac{V}{\V(\S^{d-1})}$, we get the result. %
\end{proof}
 A different approach is described in \cite{Matzke22} where  explicit and closed formulas are given for all the cases. We will only need the main term asymptotics $G(\M;p,q)=B_{\mathcal M}/(V_{\mathcal{M}}d_R(p,q)^{d-2})+\text{l.o.t.},$ with $B_{\mathcal M}$ the constant in Table \ref{table:dov}.
Other useful asymptotics are:
\begin{lem}\label{lem:assy}
For $a<<1$, we have:
\begin{align*}
V(a) =&\,\,\frac{\V(\S^{d-1})a^{d}}{d}+o(a^d),\\
v(a)=&\,\,\V(\mathbb{S}^{d-1})a^{d-1}+o(a^{d-1}),\\
K(\M,a)=&\,\, \frac{a^2}{2(d+2)V}+o(a^2),\\
\Theta(\M,a)=&\,\, \frac{dB_\mathcal M}{2V}a^{2-d}+o\left(a^{2-d}\right).
\end{align*}
The last of these equalities needs $d>2$, but the rest of them hold in all cases.
\end{lem}
\begin{proof}
All these asymptotic expansions follow from \eqref{eq:changeofv}. The first one is immediate from \eqref{eq:Volbola} and Table \ref{table:dov}. The second one follows from
$$
v(a)=\frac{dV(a)}{da}=\V(\mathbb{S}^{d-1})a^{d-1}\Omega(a)=\V(\mathbb{S}^{d-1})a^{d-1}+o(a^{d-1}).
$$
This yields the third formula of the lemma:
\begin{align*}
    V\cdot K(\M,a)=&\frac{1}{a^d}\int_0^a r^{d-1}\int_0^r u \,du\,dr +\text{l.o.t}=\frac{a^2}{2(d+2)}+o(a^2).
\end{align*}
For the last asymptotic we reason in the same way:
\begin{align}
\Theta(\mathcal M,a)=\,\,\,&\frac{1}{V(a)}\int_{q\in B(p_0,a)}G(\mathcal{M};p_0,q)dq\nonumber\\
\stackrel{\eqref{eq:changeofv}}=&\frac{\V(\mathbb S^{d-1})}{V(a)}\int_0^ar^{d-1}\Omega(r)\phi(r)\,dr\\
\stackrel{\eqref{eq:asymptGreen}}=&\frac{\V(\mathbb S^{d-1})}{V\cdot V(a)}\int_0^ar^{d-1}\frac{B_{\mathcal M}}{r^{d-2}}\,dr+\text{l.o.t},\nonumber
\end{align}
and the last claim follows.
\end{proof}

%

\section{The main technical result}

 We will generalize to harmonic manifolds an argument sketched in \cite{LN75, SM76} and described in detail in \cite[Appendix B]{Lauritsen21} for a bounded region in the plane.\begin{theorem}\label{thm:main}
Let $\mathcal M$ be a harmonic manifold and $a>0$. For any collection of $N$ points $p_1,\ldots,p_N\in\mathcal M$ we have
$$
E_{\mathcal{M}}(p_1, \ldots, p_N)\geq N\left(1-2N+ \frac{V}{V(a)}\right) K(\mathcal M,a)-N\Theta(\mathcal M,a),
$$
where, recall, $V$ is the volume of $\M$, $V(a)$ the volume of the ball of radius $a$ and the terms $K(\mathcal M,a)$ and $\Theta(\mathcal M,a)$ have been defined in \eqref{eq:defK} and \eqref{eq:defTheta}, respectively. 
\end{theorem}
\begin{proof}
Consider the following terms:
\begin{align*}
U_{BB} &=\frac{N^2}{V^2}\int_{p,q\in\mathcal{M}}G(\mathcal{M};p,q)dpdq=0,\\
U_{ij} & = G(\mathcal{M};p_i,p_j),\\
\widehat{U}_i & = -\frac{2N}{V\cdot V(a)}\int_{B(p_i,a)}\int_{\mathcal{M}} G(\mathcal{M};p,q)dpdq=0,\\
\widehat{U}_{ij} & = \frac{1}{V(a)^2} \int_{B(p_i,a)}\int_{B(p_j,a)} G(\mathcal{M};p,q)dpdq.
\end{align*}
Define $\alpha
,\gamma$ and $\delta$ by
\begin{equation*}
E_{\M}(p_1, \ldots, p_N)
=\underbrace{U_{BB}+\displaystyle\sum_{i=1}^N \widehat{U}_i+\displaystyle\sum_{i,j}\widehat{U}_{ij}}_{(\alpha)}
\underbrace{-\displaystyle\sum_{i=1}^N \widehat{U}_{ii}}_{(\gamma)}+\underbrace{\displaystyle\sum_{i\neq j}(U_{ij} -\widehat{U}_{ij})}_{(\delta)}.
\end{equation*}
Now, note that $\alpha \geqslant 0$ from Proposition \ref{positividad_alpha}, just taking
$$
\nu(p)=\frac{N}{V}-\displaystyle\sum_{i=1}^N \frac{1}{V(a)}\chi_{B(p_i,a)}(p),
$$
where $\chi_A$ is the characteristic function of the set $A$, and check that
$$
\alpha=\int_{p,q\in \mathcal{M}}G(\mathcal{M};p,q)d\nu(p)d\nu(q).
$$
We now need to find lower bounds for $\gamma$ and $\delta$. From Lemma \ref{lem:integralballprev}, we immediately have
\begin{align*}
\delta=&\sum_{i\neq j}\left(G(\mathcal{M};p_i,p_j)-\frac{1}{V(a)^2} \int_{B(p_i,a)}\int_{B(p_j,a)} G(\mathcal{M};p,q)\,dq\,dp\right)\\
\geq &\sum_{i\neq j}\left(G(\mathcal{M};p_i,p_j)-\frac{1}{V(a)} \int_{B(p_i,a)}\left(G(\mathcal{M};p,p_j)+K(\mathcal M,a) \right)dp\right)\\
\geq &\sum_{i\neq j}\left(G(\mathcal{M};p_i,p_j)-\left(G(\mathcal{M};p_i,p_j)+2K(\mathcal M,a) \right)\right)\\
=&-2N(N-1)K(\mathcal M,a),
\end{align*}
(and moreover, although we do not use it in the proof, if $B(p_i,a)\cap B(p_j,a)=\emptyset$ then the inequalities above are equalities, so for most choices of $p_i,p_j$ the inequalities above are quite sharp).

On the other hand, an elementary symmetry argument shows that
$$
\gamma=-\frac{N}{V(a)^2} \int_{p,q\in B(p_0,a)}G(\mathcal{M};p,q)dpdq,
$$
where $p_0$ is any point in $\mathcal M$. We will give a simpler formula for $\gamma$ using the fact that the integral in $\mathcal M$ of $G(\mathcal{M};p,\cdot)$ is zero:
\begin{align*}
\gamma
=&-\frac{N}{V(a)^2} \int_{p\in B(p_0,a)}\left[\int_{q\in\M}G(\mathcal{M};p,q)dq-\int_{q\notin B(p_0,a)}G(\mathcal{M};p,q)dq\right]dp\\
=&\,\,\frac{N}{V(a)^2}\int_{p\in B(p_0,a)}\int_{q\not\in B(p_0,a)}G(\mathcal{M};p,q)dqdp\\
=&\,\,\frac{N}{V(a)^2}\int_{q\not\in B(p_0,a)}\int_{p\in B(p_0,a)}G(\mathcal{M};p,q)dpdq.
\end{align*}
From Lemma \ref{lem:integralballprev}, we conclude
\begin{align*}
\gamma=&\,\,\frac{N}{V(a)}\int_{q\not\in B(p_0,a)}\left(G(\mathcal{M};p_0,q)+K(\mathcal M,a)\right)dq\\
=&\,\,\frac{N(V-V(a))}{V(a)}K(\mathcal M,a)-\frac{N}{V(a)}\int_{q\in B(p_0,a)}G(\mathcal{M};p_0,q)dq\\
=&\,\,\frac{N(V-V(a))}{V(a)}K(\mathcal M,a)-N\Theta(\mathcal{M},a).
\end{align*}
The theorem follows.
\end{proof}
\section{Proof of Theorem \ref{th:mainintro}}
Combining Lemma \ref{lem:assy} with Theorem \ref{thm:main} we have
$$
E_{\mathcal{M}}(p_1, \ldots, p_N)\geq N\left(1-2N+ \frac{dV}{\V(\S^{d-1})a^d}\right) \frac{a^2}{2(d+2)V}-N\frac{dB_\mathcal M}{2V}a^{2-d}+\text{l.o.t}.
$$
Choosing $a$ of the form $C^{1/2}N^{-1/d}$ with $C$ a constant we conclude (up to l.o.t.):
\begin{align*}
E_{\mathcal{M}}(p_1, \ldots, p_N)\geq &
-\frac{N^{2-2/d}}{V}\left(\frac{C}{d+2}+\frac{dC^{1-d/2}}{2}\left(B_{\mathcal M}-\frac{V}{(d+2)\V(\S^{d-1})}\right)\right).
\end{align*}
This last formula is maximized choosing
$$
C=\left[\frac{d(d-2)(d+2)}{4}\left(B_{\mathcal M}-\frac{V}{(d+2)\V(\S^{d-1})}\right)\right]^{\frac{2}{d}},
$$
implying, for that concrete value of $C$:
\begin{align*}
E_{\mathcal{M}}(p_1, \ldots, p_N)\geq &
-\frac{dCN^{2-2/d}}{(d^2-4)V}+\text{l.o.t},
\end{align*}
which yields the claimed lower bounds, using the values of Table \ref{table:dov} for each case.

\appendix
\section{Some properties of the Green function}
We recall some properties of $G(\mathcal M;p,q)$ which hold in any compact manifold $\mathcal M$. Green's function is in some sense the inverse of the Laplace--Beltrami operator:
\begin{prop}
If $f:\mathcal{M} \to \mathbb{R}$ is a continuous function with $\int f =0$, then
$$
u(p)=\int_{q\in\mathcal{M}} G(\mathcal{M};p,q)f(q)dq,
$$
is of class $C^2$ in $\M$ and satisfies $\Delta u=f$.
\end{prop}
\begin{proof}
See \cite[Remark 2.3]{BeltranCorralCriado2019}.
\end{proof}

The following result says that the Green function is a conditionally positive definite kernel.
\begin{prop}\label{positividad_alpha}
Let $\nu$ be any finite signed measure in $\mathcal M$ with $\nu(M)=0$. Then,
\begin{equation*}
\int_{p,q\in\mathcal{M}} G(\mathcal{M};p,q)d\nu(p)d\nu(q)\geqslant 0,
\end{equation*}
with equality if and only if $\nu=0$.
\end{prop}

\begin{proof}
See \cite[p. 166, Def. 3.2]{BeltranCorralCriado2019} and \cite[p. 175, Prop. 3.14]{BeltranCorralCriado2019}.
\end{proof}

We also have the following result \cite[p. 108, Lemma 5.3.1]{JuanCriadoRey} that gives a closed formula for the expected value of the Green function when one of its entries lives in a ball.
\begin{lem}\label{lem:integralballprev}
Let $\mathcal M=\mathbb S^n, \mathbb R\mathbb P^n, \mathbb C\mathbb P^n, \mathbb H\mathbb P^n$ or $\mathbb O\mathbb P^2$.
Then, for any $p_0,p\in\mathcal{M}$,
\begin{itemize}
\item If $d_R(p_0,p)\geq a$, then
\begin{align*}
\frac{1}{V(a)}\int_{q\in B(p_0,a)}G(\mathcal{M}; p, q)\,dq=& \,G(\mathcal{M}; p, p_0)+K(\mathcal M,a).
\end{align*}
\item If $d_R(p_0,p)< a$, then
\begin{align*}
\frac{1}{V(a)}\int_{q\in B(p_0,a)}G(\mathcal{M}; p, q)\,dq=& \,G(\mathcal{M}; p, p_0)+K(\mathcal M,a)\\
&-\frac{1}{V(a)}\int_{d(p_0,p)}^av(r)\int_{d(p_0,p)}^r \frac{du}{v(u)}dr.
\end{align*}
\end{itemize}
In particular, for any $p_0,p\in\mathcal{M}$,
$$
\frac{1}{V(a)}\int_{q\in B(p_0,a)}G(\mathcal{M}; p, q)\,dq\leq G(\mathcal{M}; p, p_0)
+ K(\mathcal M,a).
$$
\end{lem}
\begin{proof}
We sketch a proof for completeness. For the first identity, multiplying by $V(a)$ and computing the derivative with respect to $a$, it suffices to check that
\begin{equation}\label{eq:Greenidentity}
    \frac{1}{v(a)}\int_{q\in S(p_0,a)}G(\mathcal M;p,q)\,dq=G(\mathcal M;p,p_0)+\frac{1}{V}\int_0^a\frac{V(u)}{v(u)}\,du,\quad a<d_R(p,p_0).
\end{equation}
It is clear that both sides of \eqref{eq:Greenidentity} are equal as $a\to0$. We check that their derivatives also coincide. Call $F(a)$ the left--hand term in \eqref{eq:Greenidentity}. Writing it down in normal coordinates with basepoint $p_0$, we find that the derivative of the left--hand side equals
$$
F'(a)=\frac{1}{v(a)}
\int_{q\in S(p_0,a)}\nabla_{N(q)} G(\mathcal M;p,q)\,dq,
$$
where $N(q)$ is the unit vector orthogonal to $S(p_0,a)$ at $q$ and $\nabla$ is the covariant derivative. From Green's second identity, we get
$$
F'(a)=-\frac{1}{v(a)}\int_{B(p_0,a)}\Delta G(\mathcal M;p,q)\,dq=\frac{V(a)}{Vv(a)}.
$$
Hence, the derivatives at both sides of \eqref{eq:Greenidentity} are equal, proving  \eqref{eq:Greenidentity} and the first claim of the lemma in the case that $d_R(p_0,p)<a$. The case $d_R(p_0,p)=a$ follows from the continuity of both sides of the equality. Finally, if $d_R(p_0,p)=t<a$ we can still compute the derivative using Green's second identity, now to the other open set delimited by $S(p_0,a)$ and using $-N(q)$:
$$
F'(a)=\frac{1}{v(a)}\int_{\mathcal M\setminus B(p_0,a)}\Delta G(\mathcal M;p,q)\,dq=-\frac{1}{Vv(a)}(V-V(a)),\quad a>t.
$$
All in one, we have proved
\begin{align*}
    F(a)=&\,F(t)+\frac{1}{V}\int_t^a\frac{V(u)-1}{v(u)}\,du\\
    =&\,F(0)+\frac{1}{V}\int_0^t\frac{V(u)}{v(u)}\,du +
    \frac{1}{V}\int_t^a\frac{V(u)-V}{v(u)}\,du\\
    =&\,G(\mathcal M;p,p_0)+
    \frac{1}{V}\int_0^a\frac{V(u)}{v(u)}\,du -
    \int_t^a\frac{1}{v(u)}\,du.
\end{align*}
The second claim in the lemma now follows, since 
$$
\int_{q\in B(p_0,a)}G(\mathcal{M}; p, q)\,dq=\int_{0}^av(r)F(r)\,dr.
$$
\end{proof}

\section{Closed formulas for $K(\mathcal M,a)$ and $\Theta(\mathcal M,a)$}
Although we have not used them in our analysis or proofs above, in the cases $\mathcal M=\mathbb{CP}^n,\mathbb{HP}^n,\mathbb{OP}^2$ it is possible to produce exact formulas for these two functions. We summarize them in the following result.
\begin{prop}
Denoting $S=\sin a$, we have:
$$
K(\C\mathbb P^n,a)=\frac{1}{4nV S^{2n}}\left((1-S^{2n})\log (1-S^2)+ \sum_{k=1}^n\frac{S^{2k}}{k}\right),
$$
$$
\Theta(\C\P^n,a)=\frac{1}{2nV}\left(-H_{n-1}-\log S+\frac{n}{2}\sum_{k=1}^{n-1}\frac{1}{k(n-k)S^{2k}}\right),
$$
\begin{align*}
K(\HH\P^n,a)=& \frac{1}{4(2n+1)(2n(1-S^2)+1)V}\\
&\times\left[\frac{1}{S^{4n}}\left(\displaystyle\sum_{k=1}^{2n+1}\frac{S^{2k}}{k}+\log(1-S^2)\right)-(2n(1-S^2)+1)\log(1-S^2)\right],
\end{align*}
\begin{align*}
\Theta(\HH\P^n,a)
=& \frac{1}{V}\left(\frac{n}{2(2n(1-S^2)+1)}\displaystyle\sum_{k=1}^{2n-1}\frac{1}{k(k+1)(2n-k)S^{2k}}\right.\\
&\left.-\frac{H_{2n-1}}{2(2n+1)}-\frac{\log S}{2(2n+1)}-\frac{1+2(n-1)S^2}{4(2n+1)(2n(1-S^2)+1)}\right),
\end{align*}
\begin{align*}
K(\mathbb{O}\mathbb{P}^{2},a)&=
\frac{1}{1219680VS^{16}(-120S^6+396S^4-440S^2+165)}\times\\
&\big[S^2 (815640 S^{20}-1826748 S^{18}+1019480 S^{16}+3465 S^{14}+3960 S^{12}\\
&+4620 S^{10}+5544 S^8+6930 S^6+9240 S^4+13860 S^2+27720)\\
&+27720(120 S^{22}-396 S^{20}+440 S^{18}-165 S^{16}+1)\log(1-S^2)\big],
\end{align*}
 \begin{align*}
\Theta(\mathbb{O}\mathbb{P}^{2},a)
=\frac{1}{V}&\Bigg[\frac{1}{9240 S^{14} \left(-120 S^6+396 S^4-440 S^2+165\right)}\Big(
101420 S^{20}\\
& \,\,-353334 S^{18}+427500 S^{16}
-190150 S^{14}+9900 S^{12}+2310 S^{10}\\ 
&\hspace{1.5cm}+924 S^8+495 S^6+330 S^4+275 S^2+330\Big)
-\frac{1}{22}\ln S\Bigg].
\end{align*}
\end{prop}
\begin{proof}
These are all obtained directly from the definitions \eqref{eq:defK} and \eqref{eq:defTheta}, carefully computing all the indefinite integrals and using the explicit formulas given in Table \ref{table:dov}. Once computed, their correctness can be checked by automatic differentiation. 
\end{proof}

\begin{landscape}
\begin{table}[]
\begin{center}
\begin{tabular}{|c||c|c|c|c|c|}
\hline
                                & $\mathbb{S}^n$                                                    & $\mathbb{R}\mathbb{P}^n$                                         & $\mathbb{C}\mathbb{P}^{n}$                                   & $\mathbb{H}\mathbb{P}^{n}$                                     & $\mathbb{O}\mathbb{P}^{2}$                                              \rule{0pt}{0.4cm}          \\ [0.3ex] \hline\hline
$d$                                             & $n$                                                                 & $n$                                                                & $2n$                                                           & $4n$                                                             & $16$                                                                                 \rule{0pt}{0.3cm} \\ [0.3ex] \hline
$D$                                             & $\pi$                                                             & $\pi/2$                                                          & $\pi/2$                                                      & $\pi/2$                                                        & $\pi/2$                                                                 \rule{0pt}{0.3cm} \\ [0.3ex] \hline
$V$                                             & $\displaystyle\frac{2\pi^{\frac{n+1}{2}}}{\Gamma(\frac{n+1}{2})}$ & $\displaystyle\frac{\pi^{\frac{n+1}{2}}}{\Gamma(\frac{n+1}{2})}$ & $\displaystyle\frac{\pi^n}{n!}$                              & $\displaystyle\frac{\pi^{2n}}{(2n+1)!}$                        & $\displaystyle\frac{\pi^8}{1320\,\Gamma(8)}$                                     \rule{0pt}{0.7cm} \\ [2.5ex] \hline 
$V(a)/V$                                        & $I_{\sin^2\frac{a}{2}}\left(\frac{n}{2},\frac{n}{2}\right)$       & $2 I_{\sin^2\frac{a}{2}}\left(\frac{n}{2},\frac{n}{2}\right)$    & $\sin^{2n}a$                                                 & $\left(1+2n\cos^{2}a\right)\sin^{4n}a$                         & $P(a)\sin^{16}a $                   \rule{0pt}{0.4cm} \\ [1ex] \hline
$\mathrm{exp}_{(0,\ldots,0,1)}(v)$              & $\begin{pmatrix}v\frac{\sin\|v\|}{\|v\|}\\ \cos\|v\|\end{pmatrix}$               & $\begin{pmatrix}\frac{v}{\|v\|}\tan\|v\|\\1\end{pmatrix}$                      & $\begin{pmatrix}\frac{v}{\|v\|}\tan\|v\|\\1\end{pmatrix}$                  & $\begin{pmatrix}\frac{v}{\|v\|}\tan\|v\|\\1\end{pmatrix}$                    & $\begin{pmatrix}\frac{v}{\|v\|}\tan\|v\|\\1\end{pmatrix}$                              \rule{0pt}{0.7cm} \\ [2.5ex] \hline
$r^{d-1}\Omega(r)$ & $\sin^{n-1}r$                 & $\sin^{n-1}r$                & ${{\sin^{2n-1}r\cos r}}$ & ${{\sin^{4n-1}r\cos^3 r}}$ & {${{\sin^{15}r\cos^7 r}}$}  \rule{0pt}{0.4cm} \\ [1ex] \hline
$B_{\mathcal M}$&
$\frac{\sqrt{\pi}\Gamma\left(\frac{n}{2}\right)}{(n-2)\Gamma\left(\frac{n+1}{2}\right)}$&$\frac{\sqrt{\pi}\Gamma\left(\frac{n}{2}-1\right)}{4\Gamma\left(\frac{n+1}{2}\right)}$&$\displaystyle{\frac{1}{4n(n-1)}}$&$\displaystyle{\frac{1}{8n(4n^2-1)}}$&$\displaystyle{\frac{1}{36960}}$  \rule{0pt}{0.6cm} \\ [2ex]
\hline
\end{tabular}
\bigskip
\caption{The real dimension, diameter, volume, normalized volume of a ball, exponential map, volume density in the compact harmonic manifolds and the constant (without the $1/V$ factor) in the dominant term in the Green function asymptotic for $d_R(p,q)<<1$. We are using the notation $P(a)=(165-440\sin^2a+396\sin^4a-120\sin^6a)$ in the column corresponding to $\mathbb{OP}^2$. Also, $I_s(a,b)=B_s(a,b)/B(a,b)$ denotes the normalized incomplete beta function. The constant $B_{\mathcal M}$ is valid for $d>2$.}\label{table:dov}
\end{center}
\end{table}

\end{landscape}



\begin{thebibliography}{99}

\bibitem{Matzke22}
	Anderson, A., Dostert, M., Grabner, P. J., Matzke, R. W. and Stepaniuk, T. A. \textit{Riesz and
Green energy on projective spaces}. ArXiv:2204.04015 [math.CA].


\bibitem{BeltranCorralCriado2019}
		Beltrán, C., Corral, N. and G. Criado del Rey, J.
\textit{Discrete and continuous green energy on compact manifolds},
J. Approx. Theory (2019), vol. 237, 160--185.

\bibitem{BeltranEtayo18}
Beltr\'an, C. and Etayo, U. \textit{The projective ensemble and distribution of points in odd-dimensional spheres}, Constr. Approx. 48 (2018), no. 1, 163--182.


\bibitem{Pedro}
        Beltr\'an, C., Etayo, U. and L\'opez--G\'omez, P. R.
        \textit{Low energy points on the sphere and the real projective plane}. ArXiv:2206.08086 [math.CA]

 \bibitem{BMOC}
 Beltrán, C., Marzo, J. and Ortega-Cerdà, J. \textit{Energy and discrepancy of rotationally invariant determinantal point processes in high dimensional spheres}. J. Complex. 37 (2016), vol. 37, 76--109.
 
 
		
\bibitem{BeltranLizarte22}
	Beltr\'an, C. and Lizarte, F. \textit{A lower bound for the logarithmic energy on $\mathbb{S}^2$ and for the Green energy on $\mathbb{S}^n$}. ArXiv: 2205.02755 [math.CA].



\bibitem{BS18}		
		B\'etermin, L. and Sandier, E.
		\textit{Renormalized Energy and Asymptotic Expansion of Optimal Logarithmic Energy on the Sphere}, 	Constr. Approx. 47 (2018), no. 1, 39--74.		


\bibitem{BHS19}
		Borodachov, S. V., Hardin, D. P. and Saff, E. B.  \textit{Discrete energy on rectifiable sets}, Springer, New York, 2019.
		
\bibitem{Brauchart08}
		Brauchart, J. S. 
		\textit{Optimal logarithmic energy points on the unit sphere}, Math. Comp. 77 (2008), no. 263, 1599--1613. 		


\bibitem{BHS12}		
		Brauchart, J. S., Hardin, D. P. and Saff, E. B. 
		\textit{The next-order term for optimal Riesz and logarithmic energy asymptotics on the sphere}, 	Recent advances in orthogonal polynomials, special functions, and their applications, 2012, pp.	31--61.	





\bibitem{Dubickas96}		
		Dubickas, A.
		\textit{On the maximal product of distances between points on a sphere}, Liet. Mat. Rink. 36 (1996), no. 3, 303--312.	

		
\bibitem{JuanCriadoRey}
G. Criado del Rey, J. \textit{Métricas de condicionamiento y puntos bien distribuidos en variedades}. Doctoral thesis. Universidad de Cantabria (2018).		



\bibitem{Lauritsen21}
		Lauritsen, A. B. \textit{Floating wigner crystal and periodic Jellium configurations}, J. Math. Phys., 62, 083305 (2021).
		
\bibitem{LN75}
		 Lieb, E. H. and Narnhofer, H. \textit{The thermodynamic limit for jellium}, Journal of statistical physics
12 (1975), 291--310.

\bibitem{RSZ94}
		Rakhmanov, E. A., Saff, E. B. and Zhou, Y. M. 
		\textit{Minimal discrete energy on the sphere}, Math. Res. Lett. 1 (1994), no. 6, 647--662.




\bibitem{SM76}
		Sari, R. R. and Merlini, D. \textit{On the $\nu$-dimensional one-component classical plasma: the thermodynamic limit problem revisited}, Journal of statistical physics 14 (1976), 91--100.

\bibitem{Smale2000}
		Smale, S.
		\textit{Mathematical problems for the next century, Mathematics: frontiers and perspectives}, 2000, pp. 271--294.
		
\bibitem{Stefan}		
		Steinerberger, S.
		\textit{On the Logarithmic Energy of Points on $S^2$}, J. d'Analyse Math. (2022). 

\bibitem{Stefan2}
Steinerberger, S.
\textit{A Wasserstein inequality and minimal Green energy on compact manifolds},  J. Funct. Anal. 281 (2021), no. 5.


\bibitem{Wagner89}		
		Wagner, G.
		\textit{On the product of distances to a point set on a sphere}, J. Austral. Math. Soc. Ser.
		A 47 (1989), no. 3, 466--482.

\end{thebibliography}
\end{document}